\documentclass[12pt,reqno]{amsart}
\usepackage{amssymb,amsmath,amsfonts,amsthm}
\usepackage{graphicx}
\usepackage{bm}
\usepackage[all,cmtip,line]{xy}
\usepackage{array}
\usepackage{color}
\numberwithin{equation}{section}

\newtheorem{theorem}{Theorem}[section]
\newtheorem{corollary}[theorem]{Corollary}
\newtheorem{lemma}{Lemma}[section]

\newtheorem{remark}{Remark}[section]
\newtheorem{definition}{Definition}[section]

\everymath{\displaystyle}

\begin{document}

\title[Global strong solutions to  planar non-resistive  MHD equations ]
{Global strong solutions to the Cauchy problem of the planar non-resistive magnetohydrodynamic equations with large initial data}

\author{Jinkai Li}
\address{Jinkai Li, School of Mathematical Sciences, South China Normal University, Zhong Shan Avenue West 55, Guangzhou 510631, P. R. China}
\email{jklimath@m.scnu.edu.cn; jklimath@gmail.com}

\author{Mingjie Li}
\address{MingJie Li, College of Science, Minzu University of China, Beijing 100081, P. R. China}
\email{lmjmath@163.com}

\date{
May 8, 2021
}

\begin{abstract}  In this paper, we consider the Cauchy problem to the planar non-resistive magnetohydrodynamic equations without heat conductivity, and establish the global well-posedness of strong solutions with large initial data. The key ingredient of the proof is to establish the a priori estimates on the effective viscous flux and a newly introduced ``transverse effective viscous flux" vector field inducted by the transverse magnetic field. The initial density is assumed only to be uniformly bounded and of finite mass and, in particular, the vacuum and discontinuities of the density are allowed.
 \end{abstract}



\maketitle


\section{Introduction}
The full compressible magnetohydrodynamic (MHD) equations  in the Eulerian coordinates are written as (see \cite{lau84}):
\begin{equation}\label{mh}
\left\{\begin{array}{lll}
\rho_t+\textrm{div}(\rho \bm{u})=0, \\ \vspace{3mm}
\displaystyle (\rho\bm{u})_t+\textrm{div}(\rho\bm{u}\otimes\bm{u})+\nabla
P=\frac{1}{4\pi} (\nabla\times \mathbf{B})\times \mathbf{B} +\textrm{div}\Psi\bm(u)\\ \vspace{3mm}
\displaystyle \mathbf{B}_t-\nabla\times (\bm{u}\times \mathbf{B})=- \nabla\times (\nu\nabla \times \mathbf{B}),\quad  \textrm{div}\mathbf{B}=0\\
\displaystyle \bigg(\mathcal{E}+ \frac{|\mathbf{B}|^2}{8\pi}\bigg)_t+\textrm{div}\big(\bm{u}(\mathcal{E}+P)\big)\\
=\frac{1}{4\pi} \textrm{div}\big((\bm{u}\times \mathbf{B})\times \mathbf{B}\big) +\textrm{div} \bigg(\frac{\nu}{4\pi}\mathbf{B}\times (\nabla\times \mathbf{B})+\bm{u}\Psi\bm(u)+\kappa\nabla\theta\bigg).
\end{array}\right.
\end{equation}
 Here the unknowns $\rho, \bm{u}=(u_1,u_2,u_3)\in \mathbb{ R}^3, P,\mathbf{B}=(\mathbf{B}_1,\mathbf{B}_2,\mathbf{B}_3) \in \mathbb{R}^3,$ and $\theta$ denote the density, velocity,  pressure, magnetic field and temperature, respectively. $\Psi\bm(u)$ is the viscous stress tensor given by
$$\Psi\bm(u)=2\mu \mathbb{D}(\bm{u})+\lambda'\textrm{div}\bm{u}\mathbf{I}_3,$$
with $\mathbb{D}(\bm{u}):= (\nabla \bm{u} + \nabla^t \bm{u})/2$, $\mathbf{I}_3$ the $3\times 3$ identity matrix, and $\nabla^t \bm{u}$ the transpose of the matrix $\nabla \bm{u}$. $\mathcal{E}$ is the energy given by $\mathcal{E}:=\rho(e+|\bm{u}|^2/2)$ with $e$ being the internal energy, $\rho|\bm{u}|^2/2$ the kinetic energy, and $|\mathbf{B}|^2/(8\pi)$ the magnetic energy. The viscosity coefficients $\mu$ and $\lambda'$ of the flow satisfy $\mu > 0 $ and $2\mu + 3\lambda' \geq 0$.
The parameter $\nu \geq 0$ is the magnetic diffusion coefficient of the magnetic field and $\kappa \geq 0$ the heat conductivity.

This is the first paper in our series results. We start with a simple case in which we do not consider two regular terms-- the magnetic diffusion  and the heat conductivity, i.e., $\nu= 0$ and $\kappa= 0$.
In this paper, we consider the three-dimensional MHD flow with spatial variables $\mathbf{x}=(x, x_2, x_3)$, which is moving in the $x$ direction and uniform in the transverse direction $(x_2,x_3)$ (see \cite{wang03}):
\begin{equation}\label{1d}
	\begin{cases}
	\tilde{\rho}=\tilde{\rho}(x,t),\quad \tilde{p}=\tilde{p}(x,t),\\
	{\bm{u}}=(\tilde{u},\tilde{\bm{w}})(x,t),\quad\tilde{\bm{w}}=(u_2,u_3),\\
	\mathbf{B}=({b_1},\tilde{\bm{b}})(x,t),\quad\tilde{\bm{b}}=(b_2,b_3),	
	\end{cases}
\end{equation}
where $\tilde{u}$ and $b_1$ are the longitudinal velocity and longitudinal magnetic field, respectively, and $\tilde{\bm{w}}$ and $\tilde{\bm{b}}$ are the transverse velocity and transverse magnetic field, respectively.
With this special structure $(\ref{1d} )$, equations $(\ref{mh} )$ are reduced to the following system for the planar magnetohydrodynamic flows with constant longitudinal magnetic field $b_1=1$ (without loss of generality) and $\lambda =\lambda'+2\mu>0$:
\begin{equation}\label{mh1}
\begin{cases}
\tilde{\rho}_{t}+(\tilde{\rho} \tilde{u})_x=0 \\
(\tilde{\rho} \tilde{u})_{t}+(\tilde{\rho} \tilde{u}^2+\tilde{P})_x=(\lambda \tilde{u}_x)_x-\frac{1}{4\pi} \tilde{\bm{b}}\cdot\tilde{\bm{b}}_{x} \\
(\tilde{\rho}\tilde{\bm{w}})_t+(\tilde{\rho}\tilde{u}\tilde{\bm{w}})_x-\frac{1}{4\pi}\tilde{\bm{b}}_x=\mu\tilde{\bm{w}}_{xx} \\
\tilde{\bm{b}}_t+(\tilde{u}\tilde{\bm{b}})_x-\tilde{\bm{w}}_x=0\\
\tilde{P}_t+\tilde{u}\tilde{P}_x+\gamma \tilde{P}\tilde{u}_{x}=(\gamma -1)\Big(\lambda (\tilde{u}_x)^2+\mu|\tilde{\bm{w}}_{x}|^2\Big),
\end{cases}
\end{equation}	
where the pressure $\tilde{P}$ is given by

\begin{equation}
\tilde{P}=R\tilde{\rho}\tilde{\theta}=(\gamma-1)\tilde{\rho} \tilde{e}.
\end{equation}

In the sequel, we set $R=1$ without loss of generality.
We complement the system with the following initial condition:
\begin{equation}
\big(\tilde{\rho},\tilde{u},\tilde{\bm{w}},\tilde{\bm{b}},\tilde{P}\big)|_{t=0}=\big(\tilde{\rho}_0(x),\tilde{u}_0(x),\tilde{\bm{w}}_0(x),\tilde{\bm{b}}_0(x),\tilde{P}_0(x)\big)\qquad x\in \mathbb{R},
\end{equation}

There are extensive studies concerning
the theory of strong solutions for the system $(\ref{mh})$. When consider the multi-dimensional case, due to the higher nonlinearity and the degeneracy caused by the vacuum, one could get the local existence  \cite{lisu11,zhu15}
or global in time existence under some smallness condition \cite{bla18, chtan10, chhu19,
hh17, lixu13, lv16, hlx12,wuwu17}.
The low Mach number limit of the MHD system has been justified by \cite{doujiang13, huwang09, jiangj14}.

For the planar MHD system, when the initial data is discontinuous,Chen-Wang \cite{chw02} obtained the global weak solution to the free boundary problem. The existence of large strong solutions to the initial-boundary value problem for planar MHD without vacuum has been proved by Wang \cite{wang03}.
Qin-Yao \cite{qin13} showed that there is a global
solutions to free-boundary problem of planar
magnetohydrodynamic equations with radiation, general heat conductivity and large initial data.
Under the condition $\gamma-1$  is sufficiently small, Hu \cite{hu15} proved the existence of global solutions and asymptotic behavior of planar magnetohydrodynamics with large data.
Fan-Huang-Li \cite{fanh17} obtained global strong solutions to the  planar MHD system with large initial data and vacuum. All of above results deal with the boundary value problem. Recently,
Ye-Li \cite{yeli19} proved that the existence of large strong solutions to the Cauchy problem for the isentropic  planar MHD equations with magnetic diffusion by weighted estimates.

Inspired by Li-Xin \cite{lixin17, lixin19},
in this paper, we will study the global existence and uniqueness of strong solution to the Cauchy problem for the planar MHD equations $(\ref{mh1})$. The initial data is assumed to be large and may contain vacuum.

The rest of this paper is arranged as follows: in the next section, Section 2, we reformulate the system $(\ref{mh1} )$ in the Lagrangian coordinates through the flow map and state our main idea in the proof.
 Section 3 is the main part of this paper, we consider the local existence and the a priori
estimates to system $(\ref{mhd} )$ with vacuum.Consequently, we arrive at the results of Theorems 2.1 in
Section 4.

Throughout this paper, we use $C$ to denote a general positive constant which may
different from line to line.

\section{Reformulation in Lagrangian coordinates and main result}
Let $y$ be the lagrangian coordinate, and define the coordinate transform between the Lagrangian coordinate $y$ and the Euler coordinate $x$ as
$$
x=\eta(y,t),
$$
where $\eta(y, t)$ is the flow map determined by $\tilde{u}$, that is,
\begin{equation}\label{map}
\begin{cases}
\partial_{t}\eta(y,t)=\tilde{u}(\eta(y,t),t) \\
\eta(y,0)=y.
\end{cases}
\end{equation}

Define the new unknowns in the Lagrangian coordinate as
$$
(\rho,u,\bm{\omega},\bm{h},P)(y,t)=(\tilde{\rho},\tilde{u},\tilde{\omega},\tilde{\bm{b}},\tilde P)(\eta(y,t),t).
$$
Recalling the definition of $\eta$ and by straightforward calculations, one can check that
$$
\tilde{u}_x=\frac{u_y}{\eta_y}, \quad \tilde{u}_{xx}=\frac{1}{\eta_y}\left(\frac{u_y}{\eta_y}\right)_y,\quad
\tilde{u}_t+\tilde{u}\tilde{u}_x=u_t.
$$
The same relations hold for $\tilde\rho, \tilde{\bm{\omega}}, \tilde{\bm{b}},$ and $\tilde P$. Using these relations,
one can easily derive the corresponding system in the Lagrangian coordinate. However, in order to deal with the vacuum more efficiently, we introduce a new function which is the Jacobian between the Euler coordinate and the Lagrangian
coordinate:
$$J(y,t):=\eta_y(y,t).$$
One can easily check that
$$J_t=u_y.$$
Due to $(\ref{mh1})_1$ and $(\ref{mh1})_2$, it holds that
$(J\rho)_t=0,$ from which, setting $\rho|_{t=0}=\rho_0$ and since $J|_{t=0}=1$, one has
$J\rho=\rho_0.$

Using the calculations in the previous paragraph, one can rewrite system $(\ref{mh1})$ in the Lagrangian coordinate as
\begin{equation}\label{mhd}
\left\{
\begin{array}{rrcl}
&\displaystyle J_t & = & u_y,\vspace{6pt}\\
&\displaystyle \rho_0  u_t-\lambda \left(\frac{u_y}{J}\right)_y+P_y+\frac{1}{4\pi}\bm{h}\cdot\bm{h}_y& = &0,\vspace{6pt}\\
&\rho_0\bm{w}_t-\mu\left(\frac{\bm{w}_y}{J}\right)_{y}&=&\frac{1}{4\pi}\bm{h}_y, \vspace{6pt}\\
&\bm{h}_t+\frac{u_y}{J}\bm{h}-\frac{\bm{w}_y}{J}&=&0,\vspace{6pt}\\
& P_t+\gamma\frac{u_y}{J}P&=&(\gamma-1)\left(\lambda \left|\frac{u_y}{J}\right|^2+\mu\left|\frac{\bm{w}_y}{J}\right|^2\right).
\end{array}
\right.
\end{equation}

In the current paper, we consider the Cauchy problem and, thus, complement system $(\ref{mhd})$ with the following initial condition
\begin{equation}\label{ini}
\big(J,\sqrt{\rho_0}{u},\sqrt{\rho_0}{\bm{w}},{\bm{h}},{P}\big)|_{t=0}=\big(J_0,\sqrt{\rho_0}{u}_0,\sqrt{\rho_0}{\bm{w}}_0,
{\bm{h}}_0,{P}_0\big),
\end{equation}
where $J_0$ has uniform positive lower and upper bounds.
Note that by definition $J_0$ should be
identically one; however, in order to extend the local solution to be a global one, one may take some positive time $T_*$
as the initial time at which $J$ is not necessary to be identically one and, as a result, we have to deal with the
local well-posedness result with initial $J_0$ not being identically one. One may also note that the initial conditions in (\ref{ini}) are imposed on $(\sqrt{\rho_0}u,\sqrt{\rho_0}\bm{\omega})$ rather on $(u,\bm{\omega})$, in other words,
one only needs to specify the values of $(u,\bm{\omega})$ in the non-vacuum region $\{y\in\mathbb R|\rho_0(y)>0\}$.
	
Strong solutions to the Cauchy problem of system (\ref{mhd}) are defined as follows.

\begin{definition}
\label{DefLocal}
Given a positive time $T$. $\big(J,{u},{\bm{w}},{\bm{h}},{P}\big)$ is called a strong solution to system $(\ref{mhd})$, subject to $(\ref{ini})$, on $\mathbb{R} \times (0, T)$, if it has the properties
\begin{equation}\nonumber
\begin{aligned}
&\inf_{y\in\mathbb{R}, t\in (0, T )}J(y,t)>0,\\
&J-J_0\in C([0,T];H^1),\quad J_t \in L^\infty (0,T;L^{2})\cap L^2(0,T; H^1),\\
&(\sqrt{\rho_0}u,\sqrt{\rho_0}\bm{\omega})\in C([0,T];L^{2}),\quad (u_y,\bm{\omega}_y)\in L^{\infty}(0,T;L^{2})\cap L^2(0,T; H^1),\\
&(\sqrt{\rho_0}u_{t}, \sqrt{\rho_0}\bm{\omega}_t)\in L^{2}(0,T;L^{2}),\quad(\sqrt t u_{yt}, \sqrt t\bm{\omega}_{yt})\in L^2(0,T; L^2), \\
&\bm{h}\in C([0,T]; H^1),\quad  \bm{h}_t\in L^\infty(0,T; L^2)\cap L^2(0,T; H^1),\\
&P\in C([0,T];H^{1}),\quad
P_{t}\in L^{4}(0,T;L^{2})\cap L^{\frac43}(0,T; H^1),
\end{aligned}
\end{equation}
satisfies equations $(\ref{mhd})$, a.e.\,in $\mathbb{R} \times (0, T)$, and fulfills the initial condition $(\ref{ini})$.	
\end{definition}

\begin{definition}
$\big(J,{u},{\bm{w}},{\bm{h}},{P}\big)$ is called a global strong solution to system $(\ref{mhd})$, subject to $(\ref{ini})$, if it is a strong solution on $\mathbb{R} \times (0, T)$ for any positive time $T$. 	
\end{definition}

The main result in this paper reads as follows.

\begin{theorem}\label{thm}
Assume that $\rho_0\in L^1$ and $0\leq\rho_0\leq\bar\rho$ for some positive number $\bar\rho$, and that the initial data
$(J_0, u_0, \bm{\omega}_0, P_0)$ satisfies
$$
J_0\equiv1, \quad (\sqrt{\rho_0}u_0, \sqrt{\rho_0}\bm{\omega}_0, u_0', \bm{\omega}_0')\in L^2, \quad\bm{h}_0\in H^1, \quad 0\leq P_0\in L^1, \quad P_0'\in L^2.
$$
Then, there is a unique global strong solution to $(\ref{mhd})$ subject to (\ref{ini}).	
 \end{theorem}

\begin{remark}
(i) Due to the regularities of the velocity stated in Definition \ref{DefLocal}, one can transform the global solutions established in Theorem \ref{thm} in the Lagrangian coordinate back to the corresponding global solutions in the Euler coordinate. In other words, under the same condition as in Theorem \ref{thm}, the Cauchy problem to (\ref{mh1}) has a unique global strong solution.

(ii) Noticing that we only need the initial density to be nonnegative and uniformly bounded, the initial density could be very general and, in particular, it allows to have a compact support or a single point vacuum or have discontinuities.

(iii) Since $J$ has positive lower and upper bounds and $\rho=\frac{\rho_0}{J}$, one can see that
the vacuum of system (\ref{mhd}) can neither disappear nor be reformulated in the later time, and the discontinuities of the density are propagated along the characteristic lines.
\end{remark}

%
In order to prove the global existence, one has to carry out suitable a priori estimates
which are finite up to any finite time. Since system (\ref{mhd}) contains the compressible Navier-Stokes equations without heat conductivity
as a subsystem, we attempt to adopt the arguments in Li \cite{li20} and Li--Xin \cite{lixin17} to achieve these a priori estimates.
As already shown in \cite{li20,lixin17}, the effective viscous flux, i.e., the quantity $\lambda\frac{u_y}{J}-P$ there, plays a central role
in the proof. It is reasonable to believe that it is also the case for system (\ref{mhd}).
Comparing system (\ref{mhd}) with the compressible Navier-Stokes equations, it is natural
to identify the following quantity
$$
	G:=\lambda \frac{u_y}{J}-P-\frac{\bm{|h|^2}}{8\pi}
$$
as the new effective viscous flux. One can check that $G$ satisfies
 \begin{equation}
 \label{EQG}
G_t-\frac{\lambda}{J}\left(\frac{G_y}{\rho_0}\right)_y=-\gamma\frac{u_y}{J}G+
\frac{2-\gamma}{8\pi}\frac{u_y}{J}|\bm{h}|^2-(\gamma-1)\mu\left|\frac{\bm{w}_y}{J}\right|^2-\frac{\bm{h}\cdot\bm{\omega_y}}
{4\pi J},
 \end{equation}
which reduces to the one in \cite{li20,lixin17} if removing those terms involving $\bm{h}$ and $\bm{\omega}$.
The key is to get the $L^\infty(0,T; L^2)$ estimate of $G$, which is expected to be achieved by testing (\ref{EQG}) with $JG$.
In the context of the Navier-Stokes equations as considered in \cite{li20,lixin17}, where one only needs to deal with the integral
corresponding to $\frac{u_y}{J}G$, this was achieved based on the basic energy estimate and the property that $J$ has uniform positive lower
bound; in other words, the basic energy estimate and the positive lower bound of $J$ are sufficient to deal with the integral related to
$\frac{u_y}{J}G$. As for system (\ref{mhd}),
one can use the same idea as in \cite{li20,lixin17} to deal with the integral related to $\frac{u_y}{J}G$ in (\ref{EQG}); however,
the integrals related to other terms in (\ref{EQG}) can not be dealt with in the same way.
To deal with other terms in (\ref{EQG}),
in the same spirit of $G$, we introduce a new vector field $\bm{F}$, called ``transverse effect viscous flux", as
\begin{equation*}
	\bm{F}:=\mu\frac{\bm{\omega}_y}{J}+\frac{\bm{h}}{4\pi}
\end{equation*}
which satisfies
\begin{equation*}
\bm{F}_t-\frac{\mu}{J}\left(\frac{\bm{F}_y}{\rho_0}\right)_y=-\frac{u_y}{J}\bm{F}+\frac{1}{4\pi}\frac{\bm{\omega}_y}{J}. 	
\end{equation*}
One can get the $L^\infty(0,T; L^2)$ estimate of $\bm{F}$ through testing the above with $J\bm{F}$: same as before,
the basic energy estimate and the
positive lower bound of $J$ are sufficient to deal with the integral related to the term $\frac{u_y}{J}\bm{F}$; while one can get the
$L^2$ space-time integrability of $\frac{\bm{\omega}_y}{\sqrt J}$ from $(\ref{mhd})_3$ based on the basic energy estimate. With this a priori
estimate on $\bm{F}$ at hand, and combining the $L^2$ type estimate on $G$ with
$L^4$ type estimate on $\bm{h}$, one can successfully deal with all integrals related to the terms on the right hand
side of (\ref{EQG}) and, as a result, get the desired $L^\infty(0,T; L^2)$ estimate for $G$. Based on this a priori estimate, one can
further get the higher order a priori estimates and finally the global well-posedness.

\section{Local well-posedness and a priori estimates}

We start with the following local well-posedness result, which can be proved in the same way as in \cite{li20,lixin17}.

\begin{lemma}\label{lem1}
Assume that $0\leq\rho_0\leq\bar\rho$ and $\underline j\leq J_0\leq\bar j$, for some positive constants
$\bar\rho, \underline j, \bar j$, and further that
$$
(J_0',\sqrt{\rho_0}u_0, \sqrt{\rho_0}\bm{\omega}_0, u_0', \bm{\omega}_0')\in L^2, \quad(\bm{h}_0,P_0)\in H^1.
$$
Then, there is a unique strong solution $(J,u,\bm{\omega},\bm{h}, P)$ to $(\ref{mhd})$, subject to $(\ref{ini})$, on $\mathbb R\times(0,T_0)$, for some positive time $T_0$ depending only on $\mu, \lambda, \bar\rho, \underline j, \bar j$, and
$$
\|(J_0',\sqrt{\rho_0}u_0, \sqrt{\rho_0}\bm{\omega}_0, u_0', \bm{\omega}_0')\|_2+\|(\bm{h_0},P_0)\|_{H^1}.
$$
\end{lemma}

Due to Lemma \ref{lem1}, for any initial data satisfying the conditions in Theorem \ref{thm}, there is a unique local strong solution $(J, u, \bm{\omega}, \bm{h}, P)$ to system (\ref{mhd}) subject to (\ref{ini}). By iteratively applying Lemma \ref{lem1}, one can extend this solution uniquely to the maximal time of existence $T_\text{max}$.
In the rest of this section, we always assume that $(J, u, \bm{\omega}, \bm{h}, P)$ is a strong solution to system (\ref{mhd}) subject to (\ref{ini}) on $\mathbb R\times(0,T)$ for any $T\in(0,T_\text{max})$.

A series of a priori estimates for $(J, u, \bm{\omega}, \bm{h}, P)$ are established in the rest of this section, which are crucial in the next section to show the global well-posedness. In the rest of this section, it is always assumed that $J_0 \equiv 1$.

The basic energy identity is stated in the following lemma.

\begin{lemma}\label{leme0}
It holds that
$$
\int_{\mathbb{R}}\left(\frac{\rho_0 u^2}{2}+\frac{\rho_0 |\bm{\omega}|^2}{2}+\frac{J|\bm{h}|^2}{8\pi}+\frac{JP}{\gamma-1}\right)\mathrm{d}y=E_0
$$
for any $t\in(0,T)$, where
$E_0:=\int_{\mathbb{R}}\left(\frac{\rho_0 u_0^2}{2}+\frac{\rho_0 |\bm{\omega_0}|^2}{2}+\frac{|\bm{h_0}|^2}{8\pi}+\frac{P_0}{\gamma-1}\right)\mathrm{d}y.$
\end{lemma}

\begin{proof}
Multiplying $(\ref{mhd})_2$ by $u$ and integrating the resulting over $\mathbb{R}$, one gets by integration by parts that
$$
\frac{1}{2}\frac{d}{dt}\|\sqrt{\rho_0} u\|_{2}^2+\lambda \left\|\frac{u_y}{\sqrt{J}}\right\|_{2}^2=\int_{}\left(\frac{|\bm{h}|^2}{8\pi}+P\right)u_y\mathrm{d}y.
$$
Then, multiplying $(\ref{mhd})_3$ with $\bm{\omega}$ and integrating over $\mathbb{R}$, it follows from integrating by parts that
$$
\frac{1}{2}\frac{d}{dt}\|\sqrt{\rho_0} \bm{\omega}\|_{2}^2+\mu\left\|\frac{\bm{\omega}_y}{\sqrt{J}}\right\|_{2}^2=-\int_{}\frac{\bm{\omega}_y\cdot \bm{h}}{4\pi}\mathrm{d}y.
$$
Finally, multiplying $(\ref{mhd})_4$ with $J\bm{h}$, integrating over $\mathbb{R}$, and recalling that $J_t=u_y$, it holds
by integration by parts that
$$
\frac{1}{2}\frac{d}{dt}\int_{}J|\bm{h}|^2\mathrm{d}y=-\frac{1}{2}\int_{}u_y|\bm{h}|^2\mathrm{d}y
+\int_{}\bm{\omega}_y\cdot\bm{h}\mathrm{d}y.
$$
Combining the previous three equalities leads to
\begin{equation}\label{e1}
	\frac{d}{dt}\left(\frac{1}{2}\|\sqrt{\rho_0} u\|_{2}^2+\frac{1}{2}\|\sqrt{\rho_0} \bm{\omega}\|_{2}^2+\frac{\|\sqrt{J}\bm{h}\|_{2}^2}{8\pi}\right)
+\lambda\left\|\frac{u_y}{\sqrt{J}}\right\|_{2}^2+\mu\left\|\frac{\bm{\omega}_y}{\sqrt{J}}
\right\|_{2}^2=\int_{}P u_y\mathrm{d}y.
\end{equation}
In order to deal with the right hand side of $(\ref{e1})$, we multiply $(\ref{mhd})_5$ by $J$ and integrate the resulting equation over $\mathbb{R}$ to get
$$
	\frac{1}{\gamma-1}\frac{d}{dt}\int_{}PJ\mathrm{d}y+\int_{}u_y P\mathrm{d}y=\int_{}\left(\lambda\frac{(u_y)^2}{J}+\mu\frac{|\bm{w}_y|^2}{J}\right)\mathrm{d}y.
$$
Summing this with (\ref{e1}) yields
\begin{equation*}
\frac{d}{dt}\int_{}\left(\frac{1}{2}\rho_0 u^2+\frac{1}{2}\rho_0 |\bm{\omega}|^2+\frac{J|\bm{h}|^2}{8\pi}+\frac{JP}{\gamma-1}\right)\mathrm{d}y=0
\end{equation*}
from which, integrating with respect to $t$, the conclusion follows.
\end{proof}

The next lemma establishes the uniform positive lower bound of $J$.

\begin{lemma}\label{lj}
It holds that
\begin{equation}\nonumber
\inf_{(y,T)\in\mathbb R\times(0,T)}J(y,t)\ge e^{-\frac{2\sqrt{2}}{\lambda}\sqrt{\| \rho_0 \|_1 E_0}}=:\underline{J}.
\end{equation}
\end{lemma}

\begin{proof} We insert $(\ref{mhd})_1$ into $(\ref{mhd})_2$ to get
\begin{equation*}
\lambda(\ln J)_{yt}=\rho_0 u_t +P_y+\left(\frac{|\bm{h}|^2}{8\pi}\right)_y.	
\end{equation*}
Integrating the above mover $(z,y)\times(0,t)$, letting $z\rightarrow-\infty$, and noticing that $J \rightarrow 1,\bm{h} \rightarrow 0,$ and $P \rightarrow 0$, as $z \rightarrow -\infty$, one gets,
\begin{equation}
\lambda \ln J(y,t)=\int_{-\infty}^{y}\rho_0(u-u_0)\mathrm{d}y'+\int_{0}^{t}\left(P+\frac{|\bm{h}|^2}{8\pi}\right)\mathrm{d}\tau\geq \int_{-\infty}^{y}\rho_0(u-u_0)\mathrm{d}y',	\label{3.7}
\end{equation}
where the nonnegativity of $P$ has been used. By Lemma \ref{leme0}, it follows from the H\"older inequality that
$$\left|\int_{-\infty}^{y}\rho_0(u-u_0)\mathrm{d}y'\right|\leq	
2\sqrt{2\| \rho_0 \|_1 E_0}.$$
Thanks to this and recalling (\ref{3.7}), the conclusion follows.
\end{proof}

In the rest of this section, in order to simplify the presentations,
we denote by $C$ a general positive constant, which depends only on $\mu, \lambda, \bar\rho, \|\rho_0\|_1, E_0, \|(u_0',\bm{\omega}_0')\|_2+\|(\bm{h}_0,P_0)\|_{H^1},$ and $T$, is continuous in $T\in[0,\infty)$, and is finite for any finite $T$.

The next lemma gives the estimate on $\bm{\omega}$.

\begin{lemma}\label{lemom1}
It holds that
$$
\sup_{0\leq t \leq T}\|\sqrt{\rho_0}\bm{\omega}\|_{2}^2+ \int_{0}^{T}\left\|\frac{\bm{\omega}_y}{\sqrt{J}}\right\|_{2}^2\mathrm{d}t \leq C.
$$
\end{lemma}

\begin{proof}
Multiplying $(\ref{mhd})_3$ with $\bm{\omega}$ and integrating over $\mathbb{R}$, it follows from integration by parts and the Cauchy inequality that
\begin{equation*}
\frac{1}{2}\frac{d}{dt}\|\sqrt{\rho_0}\bm{\omega}\|_{2}^2+\mu\left\|\frac{\bm{\omega}_y}{\sqrt{J}}\right\|_{2}^2
=-\frac{1}{4\pi}\int_{}\bm{h}\cdot\bm{\omega}_y\mathrm{d}y
\leq \frac{\mu}{2}\left\|\frac{\bm{\omega}_y}{\sqrt{J}}\right\|_{2}^2+C\|\sqrt{J}\bm{h}\|_{2}^2
\end{equation*}
and, thus,
$$
\frac{d}{dt}\|\sqrt{\rho_0}\bm{\omega}\|_{2}^2+\mu\left\|\frac{\bm{\omega}_y}{\sqrt{J}}\right\|_{2}^2
\leq C\|\sqrt{J}\bm{h}\|_{2}^2,
$$
from which, applying the Gronwall inequality and by Lemma \ref{leme0}, the conclusion follows.
\end{proof}

The following lemma gives the estimate on the ``transverse effective viscous flux"
\begin{equation*}
	\bm{F}:=\mu\frac{\bm{\omega}_y}{J}+\frac{\bm{h}}{4\pi}.
\end{equation*}

\begin{lemma}\label{lemtr}
It holds that
$$
\sup_{0\leq t \leq T}\|\sqrt{J}\bm{F}\|_{2}^2+ \int_{0}^{T}\left(
\left\|\frac{\bm{F}_y}{\sqrt{\rho_0}}\right\|_{2}^2+\|\bm{F}\|_{\infty}^4\right)\mathrm{d}t\leq C.
$$
\end{lemma}

\begin{proof}
Recalling the definition of $\bm{F}$ and by direct calculations, one derives from $(\ref{mhd})_1, (\ref{mhd})_3,$ and $(\ref{mhd})_4$ that
\begin{equation*}
\bm{F}_t-\frac{\mu}{J}\left(\frac{\bm{F}_y}{\rho_0}\right)_y=-\frac{u_y}{J}\bm{F}+\frac{1}{4\pi}\frac{\bm{\omega}_y}{J}. 	
\end{equation*}
Multiplying this with $J\bm{F}$, integrating over $\mathbb{R}$, and integrating by parts yield
\begin{equation}\label{fes}
\begin{aligned}
&\frac{1}{2}\frac{d}{dt}\|\sqrt{J}\bm{F}\|_{2}^2+\mu \left\| \frac{\bm{F}_y}{\sqrt{\rho_0}}\right\|_{2}^2\\
=&-\frac{1}{2}\int_{} u_y |\bm{F}|^2 \mathrm{d}y+\frac{1}{4\pi}\int_{}\bm{\omega}_y\cdot \bm{F} \mathrm{d}y=\int_{} u {\bm{F}_y\cdot\bm{F}} \mathrm{d}y+\frac{1}{4\pi}\int_{} {\bm{\omega}_y}
\cdot\bm{F} \mathrm{d}y\\
\leq&\|\sqrt{\rho_0}u\|_2\left\| \frac{\bm{F}_y}{\sqrt{\rho_0}}\right\|_{2}\|\bm{F}\|_\infty+\frac{1}{4\pi}
\left\|\frac{\bm{\omega}_y}{\sqrt J}\right\|_2\|\sqrt J\bm{F}\|_2\\
\leq&C\left\| \frac{\bm{F}_y}{\sqrt{\rho_0}}\right\|_{2}\|\bm{F}\|_\infty+C
\left\|\frac{\bm{\omega}_y}{\sqrt J}\right\|_2\|\sqrt J\bm{F}\|_2,
\end{aligned}	
\end{equation}
where Lemma \ref{leme0} was used.
Recalling that $J\geq\underline{ J}$, it follows
\begin{equation}\label{fin}
	\|\bm{F}\|_{\infty}^2\leq \int \left|\partial_y|\bm{F}|^2\right|dy \leq 2\|\bm{F}\|_{2}\|\bm{F}_y\|_{2}\leq C\|\sqrt{J}\bm{F}\|_{2}\left\| \frac{\bm{F}_y}{\sqrt{\rho_0}}\right\|_{2}.
\end{equation}
Plugging the above estimate into $(\ref{fes})$, one deduces by the Young inequality that
\begin{equation*}\label{3.23}
\begin{aligned}
 \frac{1}{2}\frac{d}{dt}\|\sqrt{J}\bm{F}\|_{2}^2+\mu \left\| \frac{\bm{F}_y}{\sqrt{\rho_0}}\right\|_{2}^2
\leq&C\left\|\frac{\bm{F}_y}{\sqrt{\rho_0}}\right\|_{2}^{\frac{3}{2}}\|\sqrt{J}\bm{F}\|_{2}^{\frac{1}{2}}+\frac{1}{4\pi}\Big\|\frac{\bm{\omega}_y}{\sqrt{J}}\Big\|_{2}\|\sqrt{J}\bm{F}\|_{2}\\
\leq& \frac{\mu}{2}\left\| \frac{\bm{F}_y}{\sqrt{\rho_0}}\right\|_{2}^2 +C\left(\left\|\frac{\bm{\omega}_y}{\sqrt{J}}\right\|_{2}^2+\|\sqrt{J}\bm{F}\|_{2}^2\right)
\end{aligned}
\end{equation*}
and, thus,
$$
\frac{d}{dt}\|\sqrt{J}\bm{F}\|_{2}^2+\mu \left\| \frac{\bm{F}_y}{\sqrt{\rho_0}}\right\|_{2}^2 \leq C\left(\left\|\frac{\bm{\omega}_y}{\sqrt{J}}\right\|_{2}^2+\|\sqrt{J}\bm{F}\|_{2}^2\right),
$$
from which, by the Gronwall inequality, and applying Lemma \ref{lemom1}, it follows
$$
\sup_{0\leq t \leq T}\|\sqrt{J}\bm{F}\|_{2}^2+ \int_{0}^{T}
\left\|\frac{\bm{F}_y}{\sqrt{\rho_0}}\right\|_{2}^2 \mathrm{d}t\leq C.
$$
The estimate for $\int_{0}^{T} \|\bm{F}\|_{\infty}^4 \mathrm{d}t$ follows from the above inequality by using (\ref{fin}).
\end{proof}

As a straightforward consequence of Lemma \ref{leme0} and Lemma \ref{lemtr}, one has the following:

\begin{corollary}
  \label{cor3.1}
It holds that
$$
\sup_{0\leq t\leq T}\left\|\frac{\bm{\omega}_y}{\sqrt{J}}\right\|_2\leq C.
$$
\end{corollary}

The following lemma gives the higher integrability of the transverse magnetic field
and the effective
viscous flux
$$
	H:=|\bm{h}|^2, \qquad G:=\lambda \frac{u_y}{J}-P-\frac{\bm{|h|^2}}{8\pi}
=\lambda \frac{u_y}{J}-P-\frac{H}{8\pi}.
$$

\begin{lemma}\label{lemg}
It holds that
$$
\sup_{0\leq t \leq T}\left(\|\bm{h}\|_{4}^4+\|\sqrt JG\|_{2}^2\right)+\int_{0}^{T}\left(\left\|\frac{G_y}{\sqrt{\rho_0}}\right\|_{2}^2
+\|\bm{h}\|_{6}^6 +\|\sqrt{P}|\bm{h}|^2\|_{2}^2\right)\mathrm{d}t
\leq  C
$$
and
$$
 \int_{0}^{T}\|G\|_{\infty}^4\mathrm{d}t\leq C.
$$
\end{lemma}

\begin{proof}
By the definitions of $G, H,$ and $\bm{F}$, it follows from $(\ref{mhd})_4$ that
\begin{equation*}
\bm{h}_t=\frac{\bm{\omega}_y}{J}-\bm{h}\frac{u_y}{J}=\frac1\mu\left( \bm{F} -\frac{\bm{h}}{4\pi}\right)-\frac{\bm{h}}{\lambda}\left(G+P+\frac{H}{8\pi}\right).	
\end{equation*}
Therefore, we could obtain
\begin{equation*}
H_t=2\bm{h}\cdot\bm{h}_t=\frac{2}{\mu}\left(\bm{F}-\frac{\bm{h}}{4\pi}\right)\cdot\bm{h}-\frac{2H}{\lambda}
\left(G+P+\frac{H}{8\pi}\right),	
\end{equation*}
that is
\begin{equation}\label{he}
H_t+\frac{H^2}{4\pi \lambda}+\frac{H}{2\pi \mu}+\frac{2H P}{\lambda}=\frac{2}{\mu}\bm{F}\cdot\bm{h}-\frac{2HG}{\lambda}	.
\end{equation}
Multiplying $(\ref{he})$ with $JH$ and integrating over $\mathbb{R}$ yield
\begin{equation}\label{he2}
\begin{aligned}
\frac{1}{2}\frac{d}{dt}\int_{}JH^2\mathrm{d}y+&\int_{}\left(\frac{1}{4\pi\lambda}JH^3+\frac{1}{2\pi\mu}JH^2+\frac{2}{\lambda}J P H^2 \right)\mathrm{d}y\\
=\frac{1}{2}&\int_{}u_yH^2\mathrm{d}y+\int_{}\left(\frac{2}{\mu}J H\bm{F}\cdot\bm{h}-\frac{2}{\lambda}J G H^2 \right)\mathrm{d}y.
\end{aligned}
\end{equation}
The definition of $G$ yields
\begin{equation}\label{he3}
\frac{1}{2}\int_{}u_y H^2\mathrm{d}y=\frac{1}{2\lambda}\int_{}J H^2 \left(G+P+\frac{H}{8\pi}\right)dy.	
\end{equation}
Inserting $(\ref{he3})$ into $(\ref{he2})$ leads to
\begin{equation}\label{he4}
\begin{aligned}
\frac{1}{2}\frac{d}{dt}\|\sqrt JH\|_2^2+&\int_{}\left(\frac{3}{16\pi\lambda}JH^3+\frac{1}{2\pi\mu}JH^2+\frac{3}{2\lambda}J P H^2 \right)\mathrm{d}y\\
=&\frac{2}{\mu}\int_{} JH\bm{F}\cdot\bm{h}\mathrm{d}y-\frac{3}{2\lambda}\int J G H^2 \mathrm{d}y.
\end{aligned}	
\end{equation}
By the Cauchy inequality, it follows
\begin{equation}\label{he5}
\left|\int_{}JH\bm{F}\cdot\bm{h}dy\right| \leq \epsilon \int_{}JH^3\mathrm{d}y+C_\epsilon\int_{}J|\bm{F}|^2\mathrm{d}y	,
\end{equation}
and
\begin{equation}\label{he6}
\left|\int_{}JG H^2 \mathrm{d}y\right|\leq \epsilon \int_{}JH^3\mathrm{d}y+C_\varepsilon
\int_{} J|\bm{h}|^2G^2 \mathrm{d}y.	
\end{equation}
Taking a suitable $\epsilon$, plugging $(\ref{he5})$--$(\ref{he6})$ into $(\ref{he4})$, it follows from Lemma \ref{leme0} and Lemma \ref{lemtr} that
\begin{equation}\label{h2e}
\begin{aligned}
\frac{d}{dt}\|\sqrt{J}H\|_{2}^2+\frac{1}{4\pi\lambda}\|\sqrt{J}H^{\frac{3}{2}}\|_{2}^2
+\frac{1}{\pi\mu}\|\sqrt{J}H\|_{2}^2+\frac{3}{\lambda}\|\sqrt{JP}H\|_{2}^2&\\
\leq C(\|\sqrt J\bm{h}\|_{2}^2\|G\|_{\infty}^2+\|\sqrt{J}\bm{F}\|_{2}^2)	\leq C(1+\|G\|_\infty^2)&.
\end{aligned}
\end{equation}

By direct calculations, one can check from (\ref{mhd}) that
\begin{equation}\label{ge}
G_t-\frac{\lambda}{J}\left(\frac{G_y}{\rho_0}\right)_y=-\gamma\frac{u_y}{J}G+
\frac{2-\gamma}{8\pi}\frac{u_y}{J}|\bm{h}|^2-(\gamma-1)\mu\Big|\frac{\bm{w}_y}{J}\Big|^2-\frac{\bm{h}\cdot\bm{\omega_y}}{4\pi J}. 	
\end{equation}
Multiplying $(\ref{ge})$ with $JG$ and integrating over $\mathbb{R}$, one gets by integration by parts that
\begin{equation}\label{ge1}
\begin{aligned}
\frac{1}{2}\frac{d}{dt}\|\sqrt JG\|_2^2+\lambda\left\|\frac{G_y}{\sqrt{\rho_0}}\right\|_2^2
=&\left(\frac{1}{2}-\gamma\right)\int u_y G^2 \mathrm{d}y
-\mu(\gamma-1 )\int_{}\left|\frac{\bm{\omega }_y}{\sqrt{J }}\right|^2 G\mathrm{d}y\\
&+\frac{2-\gamma }{8\pi}\int_{}u_yHG\mathrm{d}y-\frac{1}{4\pi}\int_{}\bm{\omega}_y\cdot \bm{h} G\mathrm{d}y.
\end{aligned}
\end{equation}
The terms on the right hand side of $(\ref{ge1})$ are estimated as follows. It follows from integration by parts, the H\"older
and Young inequalities, and Lemma \ref{leme0} that
\begin{equation}\label{ge2}
\begin{aligned}
\left|\int u_y G^2 \mathrm{d}y\right|
=&2\left|\int_{} uG G_y\mathrm{d}y\right| \leq 2\|\sqrt{\rho_0}u\|_2 \left\|\frac{G_y}{\sqrt{\rho_0}}\right\|_2\|{G}\|_{\infty}\\
\leq& \epsilon\left\|\frac{G_y}{\sqrt{\rho_0}}\right\|_2^2+C_\epsilon\|G\|_\infty^2.
\end{aligned}
\end{equation}
By Lemma \ref{leme0} and Corollary \ref{cor3.1}, one deduces
\begin{equation}\label{ge3}
\begin{aligned}
\left|\int \left|\frac{\bm{\omega }_y}{\sqrt{J }}\right|^2 G \mathrm{d}y\right|
\leq \left\|\frac{\bm{\omega}_y}{\sqrt J}\right\|_2^2\|G\|_\infty\leq C(1+\|G\|_\infty^2)
\end{aligned}
\end{equation}
and
\begin{equation}\label{ge4}
\begin{aligned}
&\left|\int \bm{\omega}_y\cdot \bm{h} G \mathrm{d}y\right|\leq \|\sqrt{J}\bm{h}\|_2 \left\|\frac{\bm{\omega}_y}{\sqrt J}\right\|_2\|{G}\|_{\infty}\leq C(1+\|{G}\|_{\infty}^2).
\end{aligned}
\end{equation}
Finally, recalling the definition of $G$, by Lemma \ref{leme0}, and using the Young inequality, one deduces
\begin{equation}\label{ge5}
\begin{aligned}
&\left|\int u_yHG \mathrm{d}y\right|
=\frac{1}{\lambda }\left|\int_{} J\Big(G+P+\frac{H}{8\pi} \Big)HG\mathrm{d}y\right|\\
\leq&C\left(\|\sqrt J\bm{h}\|_2^2 \|{G}\|_{\infty}^2+ \|\sqrt{JP}H\|_2\|JP\|_1^{\frac12}\|G\|_\infty+\|\sqrt J\bm{h}\|_2\|\sqrt JH^{\frac32}\|_2\|G\|_\infty \right)\\
\leq&\epsilon\left(\|\sqrt J H^\frac32\|_2^2+ \|\sqrt{JP} H\|_2^2 \right)+C_\epsilon(1+\|{G}\|_{\infty}^2).
\end{aligned}
\end{equation}
Plugging (\ref{ge2})--(\ref{ge5}) into $(\ref{ge1})$ yields
\begin{equation}\label{gef}
\begin{aligned}
\frac{1}{2}\frac{d}{dt}\|\sqrt JG\|_2^2+\lambda\left\|\frac{G_y}{\sqrt{\rho_0}}\right\|_2^2\leq
\epsilon\left(\|\sqrt J H^\frac32\|_2^2+ \|\sqrt{JP} H\|_2^2 \right)+C_\epsilon(1+\|{G}\|_{\infty}^2)
\end{aligned}
\end{equation}
for suitably small $\epsilon$.

Adding (\ref{h2e}) with (\ref{gef}) and Choosing $\epsilon$ suitably small lead to
\begin{equation*}
  \label{ADHG}
  \begin{aligned}
    \frac{d}{dt}(\|\sqrt JH\|_2^2+\|\sqrt JG\|_2^2)&+\frac{1}{8\pi\lambda}\|\sqrt{J}H^{\frac{3}{2}}\|_{2}^2
 +\frac{1}{\lambda}\|\sqrt{JP}H\|_{2}^2+2\lambda\left\|\frac{G_y}{\sqrt{\rho_0}}\right\|_2^2\\
 \leq& C(1+\|{G}\|_{\infty}^2),
  \end{aligned}
\end{equation*}
from which noticing that
\begin{equation}\label{ginf}
	\|G\|_{\infty}^2\leq\int|\partial_y|G|^2|dy\leq 2\|G\|_{2}\|G_y\|_{2}\leq \frac{2\|\rho_0\|_{\infty}^{\frac{1}{2}}}{\sqrt{\underline J}}\|\sqrt{J}G\|_{2}\left\|\frac{G_y}{\sqrt{\rho_0}}\right\|_{2}
\end{equation}
guaranteed by $J\geq\underline J$, one obtains by the Cauchy inequality that
\begin{equation*}
  \label{ADHG}
  \begin{aligned}
    \frac{d}{dt}(\|\sqrt JH\|_2^2+\|\sqrt JG\|_2^2)&+\frac{1}{8\pi\lambda}\|\sqrt{J}H^{\frac{3}{2}}\|_{2}^2
 +\frac{1}{\lambda}\|\sqrt{JP}H\|_{2}^2+ \lambda\left\|\frac{G_y}{\sqrt{\rho_0}}\right\|_2^2\\
 \leq& C(1+\|\sqrt JG\|_2^2).
  \end{aligned}
\end{equation*}
Applying the Gronwall inequality to the above and by Lemma \ref{lj}, one gets the first conclusion.
The second conclusion follows from the first one
by using $(\ref{ginf})$.
\end{proof}

Based on Lemma \ref{lemtr} and Lemma \ref{lemg}, one can get the uniform upper bound of $(J, \bm{h}, P)$ as stated in the following lemma.

\begin{lemma}\label{lemh}
It holds that
$$
\sup_{0\leq t \leq T}(\|J\|_\infty+\|\bm{h}\|_{\infty}+\|P\|_{\infty})\leq C.
$$
\end{lemma}

\begin{proof}
Equation $(\ref{he})$ could be rewritten in terms of $G$ as
$$
H_t+\frac{1}{\lambda}\left(\frac{H}{2\sqrt{\pi}}+2G\sqrt{\pi}\right)^2+\frac{1}{\mu}\left|\frac{\bm{h}}{\sqrt{2\pi}}-\sqrt{2\pi}\bm{F}\right|^2
+\frac2\lambda HP=\frac{4\pi}{\lambda}G^2+\frac{2\pi}{\mu }|\bm{F}|^2,
$$
from which, one obtains
$$
H(y,t)\leq H_0(y)+\frac{4\pi}{\lambda}\int_0^t G^2(y,\tau)\mathrm{d} \tau +\frac{2\pi}{\mu }\int_0^t	|\bm{F}(y,\tau)|^2 \mathrm{d} \tau.
$$
Therefore, it follows from Lemma \ref{lemtr} and Lemma \ref{lemg} that
\begin{equation}\label{hi2}
\sup_{0\leq t \leq T}\|H\|_{\infty}\leq \|H_0\|_\infty + \frac{4\pi}{\lambda}\int_0^T \|G\|_\infty^2\mathrm{d} t +\frac{2\pi}{\mu }\int_0^T	\|\bm{F}\|_\infty^2 \mathrm{d} t \leq C.\end{equation}
By the definitions of $G$ and $\bm{F}$, one can rewrite $(\ref{mhd})_5$ as
\begin{equation}\label{pe}
 P_t+\frac{1}{ \lambda}\left( P+\frac{2-\gamma}{2}G+\frac{2-\gamma}{16\pi}H\right)^2
= \frac{\gamma^2}{4\lambda}\left(G+\frac{H}{8\pi}\right)^2+\frac{\gamma-1}{ \mu}\left|\bm{F}-\frac{\bm{h}}{4\pi}\right|^2.
\end{equation}
Integrating the above one over $(0, t)$ and by the Cauchy inequality yield
$$
P(y,t)\leq P_0(y)+C\int_0^t( G^2(y,\tau)+H^2(y,\tau)+|\bm{F}(y,\tau)|^2 +H(y,\tau))\mathrm{d} \tau
$$
and, thus, by Lemma \ref{lemtr}, Lemma \ref{lemg}, and (\ref{hi2}), one has
\begin{equation}\label{pes}
\sup_{0\leq t \leq T}\|P\|_{\infty}\leq\|P_0\|_\infty+C\int_0^T(\|G\|_\infty^2+\|H\|_\infty^2+\|\bm{F}\|_\infty^2+\|H\|_\infty)dt
\leq  C.
\end{equation}
Rewrite $(\ref{mhd})_1$ in terms of $G$ as
$J_t=\frac J\lambda\left(G+P+\frac{H}{8\pi}\right)$
from which one can solve
\begin{equation}\label{EXPJ}
J(y,t)=e^{\frac1\lambda\int_0^t\left(G(y,s)+P(y,s)+\frac{H(y,s)}{8\pi}\right)ds}.
\end{equation}
Then, it follows from
Lemma \ref{lemtr}, Lemma \ref{lemg}, (\ref{hi2}), and (\ref{pes}) that
$$\sup_{0\leq t\leq T}\|J\|_\infty\leq e^{\frac1\lambda\int_0^T\left(\|G\|_\infty+\|P\|_\infty+\frac{\|H\|_\infty}{8\pi}\right)dt}\leq C. $$
Combining this with (\ref{hi2}) as well as (\ref{pes}), the conclusion follows.
\end{proof}

Some higher order a priori estimates for $(J, \bm{h}, P)$ are stated in the next lemma.

\begin{lemma}\label{lemhy}
It holds that
\begin{equation*}\label{hpy}
 \sup_{0\leq t \leq T}\|(J_y,J_t, {\bm{h}_y},\bm{h}_t, {P_y})\|_2^2 +\int_0^T
 \Big(\|J_{yt}\|_2^2+\|\bm{h}_{yt}\|_2^2+\|P_t\|_2^4+\|P_{yt}\|_2^\frac43\Big)dt \leq C.
\end{equation*}
\end{lemma}

\begin{proof}
By the definitions of $G, H,$ and $F$, one can rewrite $(\ref{mhd})_4$ as
\begin{equation}\label{EQh}
\bm{h}_t=\frac1\mu\left( \bm{F} -\frac{\bm{h}}{4\pi}\right) -\frac{\bm{h}}{\lambda}\left(G+P+\frac{H}{8\pi}\right).	
\end{equation}
Differentiating the above with respect to $y$ yields
\begin{equation}\label{hye}
\partial_t \bm{h}_y+\frac{\bm{h}_y}{4\pi\mu }+\frac{P\bm{h}_y}{\lambda}+\frac{\bm{h}H_y}{8\pi\lambda}+\frac{H\bm{h}_y}{8\pi\lambda}=\frac{\bm{F}_y}{ \mu}-\frac{1}{\lambda}\Big(\bm{h}_yG+\bm{h}G_y+\bm{h}P_y\Big).
\end{equation}
Taking the inner product to the above equation with $\bm{h}_y$, it follows from the H\"older and Cauchy inequalities and Lemma \ref{lemh} that
\begin{equation}\label{hyes}
\begin{aligned}
\frac{1}{2}\frac{d}{dt}\|{\bm{h}_y}\|_{2}^2&+\frac{1}{4\pi\mu }\|\bm{h}_y\|_{2}^2
+\frac{1}{\lambda }\|\sqrt{P}\bm{h}_y\|_{2}^2
+\frac{1}{16\pi\lambda }\|{H}_y\|_{2}^2
+\frac{1}{8\pi\lambda }\|\sqrt{H}\bm{h}_y\|_{2}^2\\
&=\frac1\mu\int_{}\bm{F}_y\cdot\bm{h}_y\mathrm{d}y
-\frac{1}{\lambda}\int_{} \bm{h}_y \cdot\Big(\bm{h}_yG+\bm{h}G_y+\bm{h}P_y\Big)\mathrm{d}y	\\
&\leq C\left[\|\bm{F}_y\|_2\|\bm{h}_y\|_2+\|G\|_\infty\|\bm{h}_y\|_2^2+\|\bm{h}\|_\infty\|\bm{h}_y\|_2(\|G_y\|_2+\|P_y\|_2)\right]\\
&\leq C\left[\|\bm{F}_y\|_{2}^2+\|{G}_y\|_{2}^2+(1+\|G\|_{\infty}^2)\|\bm{h}_y\|_2^2+\|P_y\|_2^2\right].
\end{aligned}	
\end{equation}
Differentiating  equation $(\ref{pe})$ with respect to $y$ yields
\begin{equation}\label{pye}
\begin{aligned}
&\partial_t P_y+\frac{2}{ \lambda}\left( P+\frac{2-\gamma}{2}G+\frac{2-\gamma}{16\pi}H\right)\left( P_y+\frac{2-\gamma}{2}G_y+\frac{2-\gamma}{16\pi}H_y\right)\\
=&\frac{\gamma^2}{2\lambda}\left(G+\frac{H}{8\pi}\right)\left(G_y+\frac{H_y}{8\pi}\right)+\frac{2(\gamma-1)}{ \mu}\left(\bm{F}-\frac{\bm{h}}{4\pi}\right)\cdot \left(\bm{F}_y-\frac{\bm{h}_y}{4\pi}\right).
\end{aligned}
\end{equation}
Multiplying the above equation by $P_y$ and integrating over $\mathbb {R}$, it follows from the H\"older and Cauchy inequalities and Lemma \ref{lemh} that
\begin{equation}\label{py}
\begin{aligned}
 \frac{1}{2}\frac{d}{dt}\|{P_y}\|_{2}^2+\frac{2}{\lambda}\|{\sqrt{ P} P_y}\|_{2}^2
\leq&  C\|(\bm{F},G,H,\bm{h},P)\|_\infty\|({\bm{F}_y},G_y,{{H}_y},\bm{h}_y,{P_y})\|_2\|{P_y}\|_{2}\\
\leq&C \|({\bm{F}_y},{{G}_y})\|_{2}^2+(1+\|(\bm{F},G)\|_{\infty}^2)
\|({P_y},\bm{h}_y)\|_2^2.
\end{aligned}	
\end{equation}
Summing (\ref{hyes}) and (\ref{py}) and by Lemmas \ref{lemtr}--\ref{lemh}, one obtains by the Gronwall inequality
that
\begin{equation}\label{hyef}
\sup_{0\leq t \leq T}\Big(\|{\bm{h}_y}\|_{2}^2+\|{P_y}\|_{2}^2\Big)\leq C.
\end{equation}
By Lemmas \ref{leme0}, \ref{lj}, \ref{lemtr}, and \ref{lemh}, and using (\ref{hyef}), one can easily get
\begin{equation}
  \sup_{0\leq t\leq T}\Big(\|(H,\bm{h},P)\|_{H^1}+\|(\bm{F},G)\|_2\Big)
  +\int_0^T\Big(\|(\bm{F},G)\|_\infty^4+\|(\bm{F}_y,G_y)\|_2^2\Big)dt \leq C. \label{L2hPF}
\end{equation}
By (\ref{EQh}) and (\ref{hye}), it holds that
\begin{align*}
  &\sup_{0\leq t\leq T}\|\bm{h}_t\|_2^2+\int_0^T\|\bm{h}_{yt}\|_2^2dt\\
  \leq&C\int_0^T\Big(\|\bm{h}_y\|_2^2+\|(G,H,P)\|_\infty^2\|\bm{h}_y\|_2^2+\|\bm{F}_y\|_2^2
  +\|\bm{h}\|_\infty^2\|(G_y,P_y)\|_2^2\Big)dt\\
  &+C\sup_{0\leq t\leq T}\Big[\|\bm{F}\|_2^2+\|\bm{h}\|_2^2+\|\bm{h}\|_\infty^2\Big(\|G\|_2^2+\|H\|_2^2+\|P\|_2^2\Big)\Big],
\end{align*}
from which, by (\ref{L2hPF}) and Lemma \ref{lemh}, one gets
\begin{equation}
   \sup_{0\leq t\leq T}\|\bm{h}_t\|_2^2+\int_0^T\|\bm{h}_{yt}\|_2^2dt\leq C. \label{L2ht}
\end{equation}
By (\ref{pe}) and (\ref{pye}), it holds that
\begin{align*}
  &\int_0^T\Big(\|P_t\|_2^4+\|P_{yt}\|_2^\frac43\Big)dt\\
  \leq&C\int_0^T \|(\bm{F},G,H,\bm{h},P)\|_\infty^4\|(\bm{F},G,H,\bm{h},P)\|_2^4dt\\
  &+C\int_0^T\|(\bm{F},G,H,\bm{h},P)\|_\infty^{\frac43}\|(\bm{F}_y,G_y,H_y,
\bm{h}_y,P_y)\|_2^\frac43 dt \\
  \leq&C\left(\int_0^T\|(\bm{F},G,H,\bm{h},P)\|_\infty^4dt\right)^{\frac13}\left(\int_0^T\|(\bm{F}_y,G_y,H_y,
\bm{h}_y,P_y)\|_2^2dt\right)^\frac23\\
  &+C\int_0^T\|(\bm{F},G,H,\bm{h},P)\|_\infty^4\|(\bm{F},G,H,\bm{h},P)\|_2^4dt
\end{align*}
from which, by (\ref{L2hPF}) and Lemma \ref{lemh}, one gets
\begin{equation}
   \int_0^T\Big(\|P_t\|_2^4+\|P_{yt}\|_2^\frac43\Big)dt\leq C. \label{L42Pt}
\end{equation}
By direct calculations, it follows from (\ref{EXPJ}) that
\begin{align*}
  &J_t=\frac J\lambda\left(P+G+\frac{H}{8\pi}\right),\quad J_y=\frac{J}{\lambda}\int_0^t\left(P_y+G_y+\frac{H_y}{8\pi}\right)ds,\\
  &J_{yt}=\frac{J_y}{\lambda}\left(P+G+\frac{H}{8\pi}\right)+\frac J\lambda\left(P_y+G_y+\frac{H_y}{8\pi}\right).
\end{align*}
Thanks to these, it follows from (\ref{L2hPF}) and Lemma \ref{lemh} that
\begin{equation}\label{Jt}
    \sup_{0\leq t\leq T}\|(J_t,J_y)\|_2
    \leq\sup_{0\leq t\leq T}\left(\|J\|_\infty\|(G,H,P)\|_2+\|J\|_\infty\int_0^T\|(G_y,H_y,P_y)\|_2dt\right)
    \leq C
\end{equation}
and further that
\begin{equation}
  \label{JYT}
  \int_0^T\|J_{yt}\|_2^2dt\leq C\int_0^T\Big(\|J_y\|_2^2\|(G,H,P)\|_\infty^2+\|J\|_\infty^2\|(G_y,H_y,P_y)\|_2^2\Big)dt
  \leq C.
\end{equation}
Combining (\ref{hyef}) with (\ref{L2ht})--(\ref{JYT}) leads to the conclusion.
\end{proof}

Finally, we have the a priori estimates for the velocity field $(u,\bm{\omega})$.

\begin{lemma}\label{lemj}
It holds that
$$
\sup_{0\leq t \leq T} \Big(\|u_y\|_2^2+\|\bm{\omega}_y\|_2^2\Big)+\int_{0}^{T}\|({u_{yy}}, \sqrt{\rho_0} u_t, \sqrt{\rho_0} \bm{\omega}_t, {\bm{\omega}_{yy}})\|_{2}^2\mathrm{d}t \leq C.
$$
\end{lemma}

\begin{proof}
Recalling the definition of $G$, and noticing that $\rho_0 u_t = G_y$, it follows
\begin{equation*}
\begin{aligned}
&\sup_{0\leq t \leq T} \|u_y\|_2^2+\int_{0}^{ T} \|\sqrt{\rho_0} u_t\|_2^2\mathrm{d}t\\
=&\frac1\lambda\sup_{0\leq t \leq T} \left\|J\left(G+P+\frac{H}{8\pi}\right)\right\|^2_2+\int_{0}^{ T} \left\|\frac{{G}_y}{\sqrt{\rho_0}}\right\|_{2}^2\mathrm{d}t\\
\leq&C\sup_{0\leq t \leq T} \|J\|_\infty\|(G,H,P)\|_2^2+\int_{0}^{ T} \left\|\frac{{G}_y}{\sqrt{\rho_0}}\right\|_{2}^2\mathrm{d}t
\end{aligned}
\end{equation*}
from which, by (\ref{L2hPF}), Lemma \ref{lemg}, and Lemma \ref{lemh}, one obtains
\begin{equation}
  \sup_{0\leq t \leq T} \|u_y\|_2^2+\int_{0}^{ T} \|\sqrt{\rho_0} u_t\|_2^2\mathrm{d}t\leq C.
  \label{uy}
\end{equation}
Noticing that
$$
	u_{yy}=
	\frac{J_y}{\lambda}\left(G+P+\frac{H}{8\pi}\right)+\frac{J}{\lambda}\left(G_y+P_y+\frac{H_y}{8\pi}\right),
$$
it follows from (\ref{L2hPF}), Lemma \ref{lemh}, and Lemma \ref{lemhy} that
\begin{equation}\label{uyy}
\begin{aligned}
& \int_{0}^{T}\|{u_{yy}}\|_{2}^2\mathrm{d}t \leq C \int_{0}^{T}\Big(\|(P,G,H)\|^2_\infty \|J_y\|_{2}^2+\|J\|_\infty^2\|(P_y, G_y,H_y)\|_{2}^2\Big)\mathrm{d}t \leq C.
\end{aligned}
\end{equation}
Noticing that
$$
	\rho_0 \bm{\omega}_t=\bm{F}_y,\quad\bm{\omega}_y=\frac{J}{\mu}\left(\bm{F}-\frac{\bm{h}}{4\pi}\right),\quad
	\bm{\omega}_{yy}=
	\frac{J_y}{\mu }\Big(\bm{F}-\frac{\bm{h}}{4\pi}\Big)+\frac{J}{\lambda}\left(\bm{F}_y-\frac{\bm{h}_y}{4\pi}\right),
$$
it holds that
\begin{align*}
&\sup_{0\leq t\leq T}\|\bm{\omega}_y\|_2^2+\int_{0}^{T} \|(\sqrt{\rho_0} \bm{\omega}_t,~~{\bm{\omega}_{yy}}) \|_{2}^2\mathrm{d}t\\
\leq&C\int_0^T\left(\|(\bm{F},\bm{h})\|_\infty^2\|J_y\|_2^2
+\|J\|_\infty^2\|(\bm{F}_y,\bm{h}_y)\|_2^2+\left\|\frac{\bm{F}_y}{\sqrt{\rho_0}}\right\|_2^2\right)dt\\
&+C\sup_{0\leq t\leq T}\|J\|_\infty^2\|(\bm{F},\bm{h})\|_2^2,
\end{align*}
from which, by (\ref{L2hPF}), Lemma \ref{lemtr}, and Lemma \ref{lemh} one obtains that
\begin{equation}
  \sup_{0\leq t\leq T}\|\bm{\omega}_y\|_2^2+\int_{0}^{T} \|(\sqrt{\rho_0} \bm{\omega}_t,~~{\bm{\omega}_{yy}}) \|_{2}^2\mathrm{d}t\leq C. \label{wyy}
\end{equation}
Combining (\ref{uy}) with (\ref{uyy}) as well as (\ref{wyy}) leads to the conclusion. 	
\end{proof}

\section{Proof of Theorem $\ref{thm}$ }

Theorem \ref{thm} is proved as follows.

\begin{proof}[Proof of Theorem $\ref{thm}$]
By Lemma \ref{lem1}, there is a unique local strong solution, denoted by $(J, u, \bm{\omega}, \bm{h}, P)$, to system (\ref{mhd}) subject to (\ref{ini}). Besides, by iteratively applying Lemma \ref{lem1}, one can extend this solution uniquely
to the maximal time of existence $T_\text{max}$. We claim that $T_\text{max}=\infty$ and, as a result, the extended $(J, u, \bm{\omega}, \bm{h}, P)$ is a global strong solution to system (\ref{mhd}) subject to (\ref{ini}), proving the conclusion.
Assume by contradiction that $T_\text{max}<\infty$. Then, by the local well-posedness result in Lemma \ref{lem1}, it must have
\begin{equation}
  \label{TM}
  \varlimsup_{T\rightarrow T_\text{max}^-}
  \sup_{0\leq t\leq T}\left(\left(\inf_{y\in\mathbb R} J\right)^{-1}+\|(J_y,\sqrt{\rho_0}u, \sqrt{\rho_0}\bm{\omega}, u_y, \bm{\omega}_y)\|_2+\|(\bm{h},P)\|_{H^1}\right) =\infty.
\end{equation}
By Lemma \ref{lj}, it holds that $J(y,t)\geq\underline J$ for any $(y,t)\in\mathbb R\times(0,T_\text{max})$ and, thus,
\begin{equation}\label{FN1}
  \sup_{0\leq t<T_\text{max}}\left(\inf_{y\in\mathbb R} J\right)^{-1}\leq\frac{1}{\underline J}.
\end{equation}
It follows from Lemmas \ref{leme0}, \ref{lemhy}, and \ref{lemj} that
\begin{equation}\label{FN2}
  \sup_{0\leq t\leq T}\Big(\|(J_y,\sqrt{\rho_0}u, \sqrt{\rho_0}\bm{\omega}, u_y, \bm{\omega}_y)\|_2+\|(\bm{h},P)\|_{H^1}
  \Big)\leq C,
\end{equation}
for any $T\in(0,T_\text{max})$, where $C$ is a positive constant depending only on the initial data and $T_\text{max}$.
Thanks to (\ref{FN1}) and (\ref{FN2}), one gets
$$
\sup_{0\leq t<T_\text{max}}\left(\left(\inf_{y\in\mathbb R} J\right)^{-1}+\|(J_y,\sqrt{\rho_0}u, \sqrt{\rho_0}\bm{\omega}, u_y, \bm{\omega}_y)\|_2+\|(\bm{h},P)\|_{H^1}\right)<\infty,
$$
contradicting to (\ref{TM}). This contradiction implies $T_\text{max}=\infty$, proving Theorem \ref{thm}.
\end{proof}

\smallskip
{\bf Acknowledgment.}
The work of Jinkai Li
was supported in part by the National Natural Science Foundation of China (11971009, 11871005, and 11771156) and
by the Guangdong Basic and Applied Basic Research Foundation (2019A1515011621, 2020B1515310005, 2020B1515310002, and
2021A1515010247). The research of Mingjie Li is supported by the NSFC Grant No. 11671412 and 11971307.

\end{document}